\documentclass[11pt,a4paper,openany]{amsart}
\usepackage{mathrsfs}

\usepackage{amsmath,amsfonts,ams,verbatim}
\usepackage{latexsym}
\usepackage{amssymb,leftidx}
\usepackage{extarrows}
\usepackage{overpic}
\usepackage{color}
\usepackage{epsfig}
\usepackage{subfigure}

\newcommand\scc{\mathrm{sc}}
\newcommand\A{\mathcal{A}}

\makeatletter
\@addtoreset{equation}{section}\makeatother

\numberwithin{equation}{section}
\newtheorem{theorem}{Theorem}[section]
\newtheorem{proposition}[theorem]{Proposition}%[section]
\newtheorem{lemma}[theorem]{Lemma}%[section]
\theoremstyle{remark}\newtheorem{remark}[theorem]{Remark}%[section]

\begin{document}
\title[Strichartz estimate for Klein-Gordon]{Strichartz estimate and nonlinear Klein-Gordon on non-trapping scattering space}

 \author{Junyong Zhang}
\address{Department of Mathematics, Beijing Institute of Technology, Beijing, China, 100081; Department of Mathematics, Stanford University, Stanford, CA, USA, 94305} \email{zhang\_junyong@bit.edu.cn, junyongz@stanford.edu}

\author{Jiqiang Zheng}
\address{Universit\'e Nice Sophia-Antipolis, 06108 Nice Cedex 02, France}
\email{zhengjiqiang@gmail.com, zheng@unice.fr}

 \maketitle

\begin{abstract} We study the nonlinear Klein-Gordon equation on a product space $M=\R\times X$ with metric $\tilde{g}=dt^2-g$ where $g$ is the scattering metic on $X$. 
We establish the global-in-time Strichartz estimate for Klein-Gordon equation without loss of derivative by using the microlocalized  spectral measure of Laplacian on scattering manifold showed in \cite{HZ} and a Littlewood-Paley
squarefunction estimate proved in \cite{Zhang}. We prove the global
existence and scattering for a family of nonlinear Klein-Gordon equations for small initial data with minimum regularity on
this setting. 
\end{abstract}

\begin{center}
 \begin{minipage}{120mm}
 \end{minipage}
 \end{center}
\section{Introduction and Statement of Main Results}
In this paper we consider the evolution of a semilinear Klein-Gordon equations with power-type nonlinearities on a non-trapping scattering manifold. 
More specifically, we consider the following family of nonlinear Klein-Gordon equation
\begin{equation*}
\Box_{\tilde{g}} u+m^2u=F(u,D u), \quad (t,z)\in \R\times X, \quad u(0)=u_0(z), ~\partial_tu(0)=u_1(z).
\end{equation*}
Here $\Box_{\tilde{g}}=\partial_t^2-\Delta_g$ denotes the d'Alembertian in the metric $\tilde{g}=dt^2-g$ and $\Delta_g$ is the Laplacian on the manifold $X$ with scattering metric $g$ introduced by
Melrose \cite{Melrose1}. We focus on the questions of the minimum regularity for which local well-posedness and small nonlinear scattering hold true.  One of the motivation for this study is
that in low-dimensional case, we can achieve the level of regularity corresponding to a conserved quantity (such as, e.g. the energy) and thus get global existence for large initial data.

On more general class of physical manifolds,  Hintz and Vasy \cite{H, HV, HV1} studied the semilinear and quasilinear wave and Klein-Gordon equation on the physically cosmological spacetimes as solutions to 
Einstein's field equations; In particular they gave a detailed analysis of the long-time behavior of linear and nonlinear waves on Kerr-de Sitter space and non-trapping Lorentzian scattering spaces
for large regularity and small initial data.  The recent development of \cite{Vasy} allowed them to set up the analysis of the associated linear problem in a framework of Fredholm problem,
in which they used Melrose's philosophy \cite{Melrose, Melrose1} of studying differential operators $P=\Box_{\tilde{g}}$ on a non-compact space $M$ by 
compactifying $M$ to a manifold $\overline{M}$ with boundary or even corners. The concrete choice of compactification is connected to the geometric structure of $\overline{M}$ near infinity.
In our less complicated product setting, we use the same Melrose's idea with $P=\Delta_{g}$ on $X$ in the study of the spectral measure of the Laplacian \cite{H,HZ}
and then analyze the propagator $e^{it\sqrt{1-\Delta_g}}$, thus we expect better result on the lowest regularity due to the establishment of the global-in-time Strichartz estimate.

In the simplest flat Euclidean space, there is a large number of literature to study the nonlinear Klein-Gordon equation.  In the flat Euclidean space, where $X=\R^n$ and
$g_{jk}=\delta_{jk}$,  the dispersive properties of the Klein-Gordon and other dispersive equations have been proved to be powerful in the study of nonlinear problems.
The Strichartz estimate for the solution of the homogenous and inhomogeneous Klein-Gordon equation
in the form of space time integrability properties gives
\begin{equation}\label{stri}
\begin{split}
&\|u(t,z)\|_{L^q_t(I;L^r_z(\R^n))}+\|u(t,z)\|_{C(I; H^s(\R^n))}\\
&\qquad\lesssim \|u_0\|_{ H^s(\R^n)}+\|u_1\|_{
H^{s-1}(\R^n)}+\|F\|_{L^{\tilde{q}'}_t(I;L^{\tilde{r}'}_z(\R^n))},
\end{split}
\end{equation}
where the pairs $(q,r), (\tilde{q},\tilde{r})\in [2,\infty]^2$
satisfy the admissible condition for $0\leq\theta\leq1$
\begin{equation}\label{adm}
\frac{2}q+\frac{n-1+\theta}r\leq\frac{n-1+\theta}2,\quad (q,r, (n-1+\theta)/2)\neq(2,\infty,1).
\end{equation}
and the gap condition
\begin{equation}\label{scaling}
\frac1q+\frac {n+\theta}r=\frac {n+\theta}2-s=\frac1{\tilde{q}'}+\frac
{n+\theta} {\tilde{r}'}-2.
\end{equation}
We refer to Brenner \cite{B}, Ginibre-Velo\cite{GV} and Keel-Tao\cite{KT} for more details. 
In particular $\theta=0$, these estimates corresponding to wave equation serves as a tool for existence results about the nonlinear wave equation. For example, Lindblad-Sogge \cite{LS}
answered the problem of finding minimal regularity conditions on the initial data ensuring local well-posedness for semilinear wave equations. 
Similarly analogous results for the Klein-Gordon equation can be carried out as same as for the wave equation even thought the sharpness of well-posedness results is not known.
There is too many reference to cite all here, we refer the reader to \cite{IMN,NS} and the reference therein.

In view of the rich Euclidean theory due to the Strichartz estimate, it is natural to consider the corresponding equations on more general manifolds.
However it is difficult or impossible to
establish the same Strichartz-type estimates as in Euclidean space on the large class of manifold due to the influence of qualitative geometric properties.  
On asymptotically de Sitter spaces Baskin \cite{Baskin,Baskin1, Baskin2} established
a family of local (in time) weighted Strichartz estimates with derivative losses for the Klein-Gordon equation on asymptotically de Sitter spaces and 
provided a heuristic argument for the non-existence of a global dispersive estimate on these spaces.
The Strichartz estimates
are local-in-time or loss of derivatives on the compact
manifold with or without boundary, see \cite{BLP, BHS,Kap, Hart1, Tataru} and
references therein. On noncompact manifold with
nontrapping condition, one can obtain global-in-time Strichartz
estimates. For example,  the global Strichartz estimates on a exterior manifold in
$\mathbb{R}^n$ to a convex obstacle, for metrics $g$ which agrees
with the Euclidean metric outside a compact set with nontrapping
assumption,  are obtained by
Smith-Sogge \cite{HS} for odd dimension, and Burq \cite{Burq} and
Metcalfe \cite{Met} for even dimension. Blair-Ford-Marzuola
\cite{BFM} established global Strichartz estimates for the wave
equation on flat cones $C(\mathbb{S}_{\rho}^1)$ by using the
explicit representation of the fundamental solution.
Anker-Pierfelice \cite{AP} study the problem on minimal regularity condition on the initial data ensuring well-posedness for wave and Klein-Gordon on 
hyperbolic space.  On the non trapping scattering manifold, the same setting considered here, Hassell,
Tao, and Wunsch first established an $L^4_{t,z}$-Strichartz
estimate for  Schr\"odinger equation in \cite{ HTW1} and then they \cite{HTW} extended the
estimate to full admissible local-in-time Strichartz estimate except endpoint
$q=2$. More recently, Hassell-Zhang \cite{HZ} improved the local-in-time one to global-in-time one and fixed the endpoint $q=2$ by analyzing the microlocalized spectral measure.
Following this Zhang \cite{Zhang} extended the global-in-time result for the wave equation. Bouclet-Mizutani \cite{BM} generalized the Schr\"odinger result to the setting with mild trapping and with more general ends.

In this paper, we will establish the global-in-time Strichartz estimate for the Klein-Gordon and apply it to 
study the minimal regularity problem for nonlinear Klein-Gordon on the non-trapping scattering manifold (asymptotically conic manifold)
which is the same as in \cite{HTW, HZ, Zhang} including the asymptotically Euclidean space.  The scattering manifold means that $X$ can be compactified to a manifold with boundary
$\overline{X}$ such that $g$ becomes a scattering metric on $\overline{X}$; see more about this next section. For geometric reasons, we expect the same dispersive properties of Klein-Gordon as in the Euclidean setting.
The key ingredient is to establish global-in-time Strichartz estimate for Klein-Gordon. It is known that Klein-Gordon behaviors like Schr\"odinger at low frequency and wave equation at high frequency. 
As same as the Euclidean space, we introduce a parameter $\theta$ for Klein-Gordon admissible pair which is wave admissible at $\theta=0$ and Schr\"odinger pair at $\theta=1$.
More precisely, we have the result about Strichartz estimates in the following.

Let $ H^s(X)={(1-\Delta_g)}^{-\frac{s}2}L^2(X)$ be
the inhomogeneous Sobolev space over $X$. Throughout this paper,
pairs of conjugate indices are written as $r, r'$, where
$\frac{1}r+\frac1{r'}=1$ with $1\leq r\leq\infty$.

\begin{theorem}[Global-in-time Strichartz estimate]\label{Strichartz} Let $(X,g)$ be non-trapping scattering manifold of dimension
$n\geq3$.  Suppose that $u$ is the
solution to the Cauchy problem
\begin{equation}\label{eq}
\begin{cases}
\partial_{t}^2u-\Delta_g u+u=F(t,z), \quad (t,z)\in I\times X; \\ u(0)=u_0(z),
~\partial_tu(0)=u_1(z),
\end{cases}
\end{equation}
for some initial data $u_0\in  H^{s}, u_1\in  H^{s-1}$, and
the time interval $I\subseteq\R$, then
\begin{equation}\label{stri}
\begin{split}
&\|u(t,z)\|_{L^q_t(I;L^r_z(X))}+\|u(t,z)\|_{C(I; H^s(X))}\\
&\qquad\lesssim \|u_0\|_{ H^s(X)}+\|u_1\|_{
H^{s-1}(X)}+\|F\|_{L^{\tilde{q}'}_t(I;L^{\tilde{r}'}_z(X))},
\end{split}
\end{equation}
where the pairs $(q,r), (\tilde{q},\tilde{r})\in [2,\infty]^2$
satisfy the KG-admissible condition with $0\leq\theta\leq1$
\begin{equation}\label{adm}
\frac{2}q+\frac{n-1+\theta}r\leq\frac{n-1+\theta}2,\quad (q,r,n,\theta)\neq(2,\infty,3,0).
\end{equation}
and the gap condition
\begin{equation}\label{scaling}
\frac1q+\frac {n+\theta}r=\frac {n+\theta}2-s=\frac1{\tilde{q}'}+\frac
{n+\theta}{\tilde{r}'}-2.
\end{equation}
\end{theorem}

\begin{remark}We remark that the estimates here are the same as the Strichartz estimates for Klein-Gordon on Euclidean space which are global-in-time and have no loss of
derivatives. 
\end{remark}

We sketch the proof as follows. As same as \cite{HZ, Zhang}, our strategy is to use the abstract
Strichartz estimate proved in Keel-Tao \cite{KT}.  Thus,
with $U(t)$ denoting the (abstract) propagator, we need to show
uniform $L^2\rightarrow L^2$ estimate for $U(t)$, and
$L^1\rightarrow L^\infty$ type dispersive estimate on the $U(t)
U(s)^*$ with a bound of the form $O((1+|t-s|)^{-(n-1+\theta)/2})$ with $0\leq\theta\leq 1$. In the flat
Euclidean setting, the estimates are usually obtained by using stationary phase argument. One point is to write the propagator in the form of oscillatory integral. Since the
Laplacian in our general setting is degenerate when it is close to the compactified boundary, the formulate turns out to be more complicated. On the other hand, the conjugate point occuring in
this non-flat setting may lead to the failure of the dispersive estimate. For example,
\cite{HW1} showed that the Schr\"odinger propagator $e^{it\Delta_g}$ failed
to satisfy such a dispersive estimate at any pair of conjugate
points $(z,z') \in X \times X$ (i.e. pairs $(z,z')$
where a geodesic emanating from $z$ has a conjugate point at $z'$).
Fortunately, we can localize the propagator to separate the conjugating points
and write the propagator in a form of oscillatory integral by using a microlocalized spectral measure. The microlocalized spectral measure $Q_j(\lambda)
dE_{\sqrt{-\Delta_g}}(\lambda) Q_j(\lambda)^*$ constructed in \cite{HZ} not only has a size
estimate in but captures its oscillatory
behavior, where $Q_j(\lambda)$ is a member of a
partition of the identity operator in $L^2(X)$.  
In the stationary phase argument, the Klein-Gordon multiplier $e^{it\sqrt{1+\lambda^2}}$ behaviors like wave at high frequency and Schr\"odinger at low frequency.
We establish the dispersive estimate with norm $O((1+|t-s|)^{-n/2})$ at low frequency and $O((1+|t-s|)^{-(n-1+\theta)/2})$ at high frequency.
We finally show the Strichartz estimate from a frequency-localized Strichartz estimate by a square function estimate proved in \cite{Zhang}. The inhomogeneous Strichartz estimates follow from the
homogeneous estimates and the Christ-Kiselev lemma.  \vspace{0.2cm}

Having the Strichartz estimate, we first consider the well-posedness and
nonlinear scattering problem of the Cauchy problem on this setting
\begin{equation}\label{neq}
\begin{cases}
\partial_t^2 u-\Delta_g u+u=\pm |u|^{p-1}u,\qquad
(t,z)\in\R\times X,\\
u(t,z)|_{t=0}=u_0(z),\quad
\partial_tu(t,z)|_{t=0}=u_1(z).
\end{cases}
\end{equation}
In the case of flat Euclidean space, there are many results on the
understanding of the global existence and scattering. We refer the
readers to \cite{LS,Sogge} and references therein.
We here are mostly interested in the range of exponents
$p\in[p_{\text{conf}},1+\frac4{n-2}]$ and the initial data is in $
H^{s_c}(X)\times  H^{s_c-1}(X)$, where
$p_{\text{conf}}=1+\frac4{n}$  and $s_c=\frac
n2-\frac2{p-1}$. The critical power $1+4/n$, which is different from wave equation's $1+\frac4{n-1}$ in \cite{Zhang}, is related to the dispersive estimate decay rate  
and it appears in the theorem because of the fact that NLKG is conformally invariant only if $F(u)=u^{1+4/n}$; see \cite{Strauss}. The other
power $1+4/(n-2)$ is related to the energy-critical index.\vspace{0.2cm}

Our first main result is about the well-posedness and nonlinear scattering with smallest regularity.

\begin{theorem}\label{thm1}  Let $(X,g)$ be a non-trapping scattering manifold of dimension
$n\geq3$.
Then if $(u_0,u_1)\in  H^{s_c}(X)\times H^{s_c-1}(X)$ and $p\in[1+\frac4{n-1},1+\frac4{n-2}]$, there exist $T>0$ and  a unique solution $u$ to \eqref{neq}
satisfying
\begin{equation}\label{lwp} u\in C_t([0,T];  H^{s_c}(X))\cap
L^{q_0}([0,T];L^{q_0}(X)),\end{equation} where
$q_0=(p-1)(n+1)/2$. In addition, if there is a small constant
$\epsilon(p)$ such that
\begin{equation}\label{small}
\|u_0\|_{ H^{s_c}}+\|u_1\|_{ H^{s_c-1}}<\epsilon(p),
\end{equation}
then there is a unique global and scattering solution $u$ to
\eqref{neq} satisfying
\begin{equation}\label{gwp} u\in C_t(\R; H^{s_c}(X))\cap
L^{q_0}(\R;L^{q_0}(X)).\end{equation}
Furthermore if $(u_0,u_1)\in  H^{1}(X)\times L^{2}(X)$ and $p\in[p_{\mathrm{conf}},1+\frac4{n-2})$, there exists a global solution to \eqref{neq} with minus sign in the nonlinearity.
\end{theorem}
\begin{remark} On local well posedenss and small scattering with the minimal regularity result, we have to restrict ourself with $p\in[1+\frac4{n-1},1+\frac4{n-2}]$, that is, $s_c\geq1/2$
 and we can extend the similar result to $p\in[p_{\mathrm{conf}},1+\frac4{n-1}]$ if 
$(u_0,u_1)\in  H^{s}(X)\times H^{s-1}(X)$ with $s\geq1/2$.
\end{remark}
\vspace{0.2cm}

We next specially consider the well-posedness of the following Yang-Mills-type equations on this setting with dimension $n=3$.
\begin{equation}\label{YM}
\begin{cases}
\partial_t^2 u-\Delta_g u+u=uDu+ |u|^{2}u,\qquad
(t,z)\in\R\times X,\\
u(t,z)|_{t=0}=u_0(z)\in H^s(X),\quad
\partial_tu(t,z)|_{t=0}=u_1(z)\in H^{s-1}(X).
\end{cases}
\end{equation}
The derivative $Du$ is measured relative to the metric structure, more precisely, $D$ is a first order scattering differential operator. If dropping the linear term $u$, this equation has the same scaling as cubic NLW, but is more difficult technically because of the derivative term $uDu$.
In the Euclidean space, this Yang-Mills-type wave equation was proved to be local well-posedness when $s>1$ in \cite{PS} and was showed to be ill-posedness when $s\leq 1$ in \cite{L}.

\begin{theorem}\label{thm2}  Let $(X,g)$ be a non-trapping scattering manifold of dimension
$n=3$ and let $0<\delta\ll1$. Suppose 
$(u_0,u_1)\in  H^{1+\delta}(X)\times H^{\delta}(X)$,
then there exist $T>0$ and  a unique solution $u$ to \eqref{neq}
satisfying
\begin{equation}\label{thm2:lwp} u\in C_t([0,T];  H^{1+\delta}(X))\cap
L^{2}([0,T];L^{\infty}(X)).\end{equation}
\end{theorem}
\vspace{0.2cm}

This paper is organized as follows. In Section 2 we review the background of scattering manifold, the
results of the microlocalized spectral measure for the Laplacian and the square
function inequalities on this setting. Section 3 is devoted to the
proofs of the microlocalized dispersive estimates and
$L^2$-estimates. In Section 4, we prove the homogeneous and
inhomogeneous Strichartz estimates.  Finally, we apply the
Strichartz estimates to show Theorem \ref{thm1} and Theorem \ref{thm2}.\vspace{0.1cm}

{\bf Acknowledgments:}\quad  The authors would like to thank Andrew Hassell, Changxing Miao and Andras Vasy for their
helpful discussions and encouragement.  This work was supported by China Scholarship Council, National Natural
Science Foundation of China (11401024), and the European
Research Council, ERC-2012-ADG, project number 320845:  Semi
Classical Analysis of Partial Differential Equations.\vspace{0.2cm}

\section{Some analysis tools on scattering manifold}
In this section, we briefly recall the key elements of the
microlocalized spectral measure and a fundamental Littlewood-Paley squarefunction estimate. The first one was constructed by Hassell-Zhang \cite{HZ} to capture both its size and the
oscillatory behavior. The second one was proved in \cite{Zhang}. \vspace{0.2cm}

\subsection{Geometry setting}
Let us recall the manifold with scattering metric introduced by Melrose \cite{Melrose1}. There is many work to analyze the Laplacian operator on the scattering manifold, that is, asymptotically conic geometric setting;
see \cite{GHS1,GHS2,HW1,HTW,HZ}. Let $(X,g)$ be a complete
noncompact Riemannian manifold of dimension $n\geq2$ with one end,
diffeomorphic to $(0,\infty)\times Y$ where $Y$ is a smooth compact
connected manifold without boundary. Moreover, we assume
$(X,g)$ is scattering manifold which
means that $X$ allows a compactification $\overline{X}$ with
boundary, with $\partial {\overline{X}}=Y$, such that the metric $g$ becomes an
asymptotically conic metric on $\overline{X}$. In details, the metric $g$ in a
collar neighborhood $[0,\epsilon)_x\times \partial \overline{X}$ near $Y$ takes
the form of
\begin{equation}\label{metric}
g=\frac{\mathrm{d}x^2}{x^4}+\frac{h(x)}{x^2}=\frac{\mathrm{d}x^2}{x^4}+\frac{\sum
h_{jk}(x,y)dy^jdy^k}{x^2},
\end{equation}
where $x\in C^{\infty}(\overline{X})$ is a boundary defining function for
$\partial \overline{X}$ and $h$ is a smooth family of metrics on $Y$. Here we
use $y=(y_1,\cdots,y_{n-1})$ for local coordinates on $Y=\partial
M$, and the local coordinates $(x,y)$ on $\overline{X}$ near $\partial\overline{X}$. Away
from $\partial \overline{X}$, we use $z=(z_1,\cdots,z_n)$ to denote the local
coordinates. If $h_{jk}(x,y)=h_{jk}(y)$ is independent of $x$, we say
$\overline{X}$ is perfectly conic near infinity. Moreover if every geodesic
$z(s)$ in $\overline{X}$ reaches $Y$ as $s\rightarrow\pm\infty$, we say $\overline{X}$ is
nontrapping. The function $r:=1/x$ near $x=0$ can be thought of as a
``radial" variable near infinity and $y$ can be regarded as the $n-1$
``angular" variables; the metric is asymptotic to the exact
conic metric $((0,\infty)_r\times Y, dr^2+r^2h(0))$ as $r\rightarrow
\infty$. The Euclidean space $X=\mathbb{R}^n$ is an example of
an asymptotically conic manifold with $Y=\mathbb{S}^{n-1}$ and the
standard metric. 

\subsection{The Laplacian on scattering manifold}
Our setting is on the scattering manifold, we turn to the concepts of ``scattering geometry". For a full discussion of scattering geometry, we refer the reader to Melrose \cite{Melrose1}.
The space of \emph{sc-vector fields} is defined as $\mathcal{V}_{\mathrm{sc}}\overline{X})=x\mathcal{V}_b(\overline{X})$, 
where $\mathcal{V}_b(\overline{X})$ is the Lie algebra of all smooth vector fields on $\overline{X}\overline{X}$ which are tangent to the boundary.
The sc-vector field also forms a Lie algebra.
These sc-vector field can be realized as the sections of a
vector bundle $\leftidx{^{\scc}}{T\overline{X}}$, called the sc-tangent bundle.
That means $\mathcal{V}_{\scc}(\overline{X})=\mathcal{C}^\infty(\overline{X};\leftidx{^{\scc}}{T\overline{X}})$,
i.e. $\mathcal{V}_{\scc}(\overline{X})$ is a space of sections of
$\leftidx{^{\scc}}{T\overline{X}}$ the sc-tangent bundle over $\overline{X}$.  Using above notation in which
$x$ is the boundary defining function of $\overline{X}$ and $y$ are coordinates
in $\partial \overline{X}$, we have
\begin{equation*}\mathcal{V}_{\scc}(\overline{X})=\begin{cases}\mathcal{V}, ~\text{i.e. all $\mathcal{C}^\infty$-vector
fields}, &\text{in the interior} ~\overline{X};\\
\text{span}\{x^2\partial_x, x\partial_{y_1}\cdots,
x\partial_{y_{n-1}}\}, & \text{near the boundary} ~\partial \overline{X}.
\end{cases}
\end{equation*}

We denote by $\text{Diff}_{\text{sc}}^*(\overline{X})$ the `enveloping algebra' of $\mathcal{V}_{\scc}(\overline{X})$, meaning the ring of differential operator on $\mathcal{C}^\infty(\overline{X})$ generated by $\mathcal{V}_{\scc}(\overline{X})$ and $\mathcal{C}^\infty(\overline{X})$.
In particular, near the boundary $\partial \overline{X}$, the $k$-order scattering differential operator is given by
\begin{equation*}\text{Diff}_{\text{sc}}^k(\overline{X})=\Big\{\A: \A=\sum_{j+|\alpha|\leq
k}a_{j\alpha}(x,y)(x^2\partial_x)^j(x\partial_y)^\alpha,
a_{j\alpha}\in C^\infty(\overline{X}) \Big\}.
\end{equation*}
If $\alpha\in\R$, the $\alpha$ b-density bundle, denoted by
$\leftidx{^{\scc}}\Omega^\alpha$, is defined by
\begin{equation*}
\leftidx{^{\scc}}\Omega^\alpha X=\bigcup_{p\in
\overline{X}}\Omega^\alpha(\leftidx{^{\scc}}{T_p\overline{X}}).
\end{equation*}
In particular $\alpha=1/2$, it is convenient to regard such operators as acting on $\scc$-half densities, that is, multiples of a half-density 
taking the form $\left|\frac{dx}{x^2}\frac{dy_1}{x}\cdots \frac{dy_{n-1}}{x}\right|^{1/2}$. Correspondingly, the Schwartz kernels of such operators can be written as 
a distribution tensored with a  scattering half density in each of the left and right variables.

Define $\leftidx{^{\scc}}{T^*\overline{X}}$, the scattering cotangent bundle over
$\overline{X}$, to be the dual vector bundle to $\leftidx{^{\scc}}{T\overline{X}}$. 
Locally near the boundary, in the coordinate $(x,y)$, we have
\begin{equation*}\leftidx{^{\scc}}{T^*\overline{X}}=
\text{span}\Big\{\frac{\mathrm{d}x}{x^2},
\frac{\mathrm{d}{y}}{x}\Big\}=\text{span}\Big\{\mathrm{d}\big(\frac1x\big),
\frac{\mathrm{d}{y}}{x}\Big\}.
\end{equation*} Thus for any $ \alpha\in \leftidx{^{\scc}}{T^*\overline{X}}$ can
be written
\begin{equation*}\alpha=\tau\mathrm{d}\big(\frac1x\big)
+\mu\cdot\frac{\mathrm{d}{y}}{x},
\end{equation*}
and this gives a linear coordinates $(\tau,\mu)\in\R\times\R^{n-1}$
on each fiber of $\leftidx{^{\scc}}{T^*\overline{X}}$. Thus this also gives a
linear coordinates $(x,y; \tau,\mu)$ on
$\leftidx{^{\scc}}{T^*\overline{X}}$ near the boundary $\partial \overline{X}$. On the other
hand, if $(\xi,\eta)$ is the dual cotangent variables to $(x,y)$,
then
\begin{equation*}
\alpha=\xi\mathrm{d}x+\eta\cdot\mathrm{d}y
\end{equation*}
which implies $\tau=x^2\xi, \mu=x\eta$. We say $(\tau,\mu)$ as
rescaled cotangent variables. Hence this space of operators can be microlocalized by introducing \emph{ scattering pseudodifferential operators} which are formally objects given by $b(x,y,x^2\partial_x, x\partial_y)$
with $b(x,y,\tau,\mu)$ a Kohn-Nirenberg symbol on the bundle $\leftidx{^{\scc}}{T^*\overline{X}}$.

In the above coordinates,  the Laplacian can be written
\begin{equation}
\Delta_g=\sum_{j,k=1}^n\frac{1}{\sqrt{|g|}}\partial_j g^{j,k} \sqrt{|g|}\partial_k
\end{equation}
where $|g|$ is the determinant of the metric $g_{jk}$. To compare with the Euclidean space near the boundary, we write the metric near the boundary in 
the form $dr^2+r^2h(x,y, dy, r^{-2}dr)$ with respect to the local coordinates $r=1/x$ and $y$. Then the metric components satisfy
\begin{equation}
\begin{split}
&g_{00}=1+O(r^{-2}), g_{0,j}=O(1), g_{kj}=r^2(\tilde{h}_{k,j}+O(r^{-1}))\\&
g^{00}=1+O(r^{-2}), g^{0,j}=O(r^{-2}), g^{kj}=r^{-2}(\tilde{h}^{k,j}+O(r^{-1}))
\end{split}
\end{equation}
where $\tilde{h}$ is the induced metric on the boundary. Note that the cross term, with $j=0$ and $k\neq0$ or $j\neq0$ and $k=0$ vanish as $x^3$
when expressed in terms of $x\partial_x$ and $\partial_y$ (the components in $\mathcal{V}_b$).
Hence near the boundary we write 
\begin{equation}
\Delta_g=(x^2\partial_x)^2+(n-1)x^3\partial_x+x^2\Delta_h+x^3 \mathrm{Diff}_b^2(\overline{X})
\end{equation}
where $\mathrm{Diff}_b^2$ is the second order differential b-operator. In this sense, the Laplacian on this setting is a sc-differential operator.
To see more results about its resolvent and calculus, we refer to \cite{HV1,  Melrose}.

\subsection{The microlocalized spectral measure}In the free Euclidean space, the Klein-Gordon propagator can be written in
an explicit formula by using the Fourier transform, but in our setting it turns out to be quite
complicated. From the results of \cite{GHS1,HW}, we have known that
the Schwartz kernel of the spectral measure can be described as a
Legendrian distribution on the compactification of the space
$\overline{X}\times \overline{X}$ uniformly with respect to the spectral parameter
$\lambda$. As pointed out in introduction, we really need to choose
an operator partition of unity to microlocalize the spectral measure
such that the spectral measure can be expressed in a formula
capturing not only the size also the oscillatory behavior. This was
constructed and proved in \cite{HZ}. For convenience, we recall it here.

\begin{proposition}
\label{prop:localized spectral measure} Let $(X,g)$ and
$\mathrm{H}=-\Delta_g$ be in Theorem \ref{Strichartz}. Then there exists a $\lambda$-dependent  scattering pesudodifferential operator partition of unity on
$L^2(M)$
$$
\mathrm{Id}=\sum_{j=1}^{N}Q_j(\lambda),
$$
with $N$ independent of $\lambda$,
such that for each $1 \leq j \leq N$ we can write
\begin{equation}\label{beanQ}\begin{gathered}
(Q_j(\lambda)dE_{\sqrt{\mathrm{H}}}(\lambda)Q_j^*(\lambda))(z,z')=\lambda^{n-1} \Big(  \sum_{\pm} e^{\pm
i\lambda d(z,z')}a_\pm(\lambda,z,z') +  b(\lambda, z, z') \Big),
\end{gathered}\end{equation}
with estimates
\begin{equation}\label{bean}\begin{gathered}
\big|\partial_\lambda^\alpha a_\pm(\lambda,z,z') \big|\leq C_\alpha
\lambda^{-\alpha}(1+\lambda d(z,z'))^{-\frac{n-1}2},
\end{gathered}\end{equation}
\begin{equation}\label{beans}\begin{gathered}
\big| \partial_\lambda^\alpha b(\lambda,z,z') \big|\leq C_{\alpha, M}
\lambda^{-\alpha}(1+\lambda d(z,z'))^{-K} \text{ for any } K.
\end{gathered}\end{equation}
Here $d(\cdot, \cdot)$ is the Riemannian distance on $X$.

\end{proposition}

From this proposition, we can exploit the oscillations both in the
multiplier $e^{i(t-s)\sqrt{1+\lambda^2}}$ and in $e^{\pm i\lambda d(z,z')}$ to
obtain the required dispersive estimate for the $TT^*$ version of
the microlocalized propagator.

\subsection{The Littlewood-Paley squarefunction estimate} In \cite{Zhang}, we showed the Gaussian upper bounds on the heat kernel by
using the local-in-time heat kernel bounds in Cheng-Li-Yau
\cite{CLY}, and Guillarmou-Hassell-Sikora's \cite{GHS2} restriction
estimate for low frequency. Hence we finally proved the Littlewood-Paley squarefunction
estimate on this setting by using a spectral multiplier estimate
in Alexopoulos \cite{Alex} and Stein's \cite{Stein} classical
argument involving Rademacher functions. Now we recall the result here for convenience. 

Let $\phi\in C_0^\infty(\mathbb{R}\setminus\{0\})$ take values in
$[0,1]$ and be supported in $[1/2,2]$ such that
\begin{equation}\label{dp}
1=\sum_{j\in\Z}\phi(2^{-j}\lambda),\quad\lambda>0.
\end{equation}
Define $\phi_0(\lambda)=\sum_{j\leq0}\phi(2^{-j}\lambda)$.
The result about the Littlewood-Paley squarefunction estimate
reads as follows:
\begin{proposition}\label{prop:square} Let $(X,g)$ be a scattering
manifold, trapping or not, and $\mathrm{H}=-\Delta_g$ is the
Laplace-Beltrami operator on $(X,g)$. Then for $1<p<\infty$,
there exist constants $c_p$ and $C_p$ depending on $p$ such that
\begin{equation}\label{square}
c_p\|f\|_{L^p(X)}\leq
\big\|\big(\sum_{j\in\Z}|\phi(2^{-j}\sqrt{\mathrm{H}})f|^2\big)^{\frac12}\big\|_{L^p(X)}\leq
C_p\|f\|_{L^p(X)}.
\end{equation}
\end{proposition}

One important application of the traditional Littlewood--Paley theory is the proof of Leibniz (=product) and chain rules for
differential operators of non-integer order.  For example, if $1<p,p_j<\infty$ with $j=1,\cdots 4$ and $s>0$, then
$$
\| f g \|_{H^{s,p}(\R^n)} \lesssim \| f \|_{H^{s,p_1}(\R^n)} \| g \|_{L^{p_2}(\R^n)} + \| f \|_{L^{p_3}(\R^n)}\| g \|_{H^{s,p_4}(\R^n)}
$$
whenever $\frac1p=\frac1{p_1}+\frac1{p_2}=\frac1{p_3}+\frac1{p_4}$.  For a textbook presentation of these theorems and original references, see \cite{Taylor}.
The Leibniz chain rules is a basic tool in the proof of well-posedness. Since we have heat kernel estimate with Gaussian upper bounds and the Littlewood-Paley squarefunction estimate,
the Leibniz chain rules can be obtained by similar argument in Euclidean space and it also was proved in \cite[Theorem 27]{CRT}. We record here
\begin{proposition}\label{frule} Let $ H^{s,p}(X)={(1-\Delta_g)}^{-\frac{s}2}L^p(X)$ be
the inhomogeneous Sobolev space over $X$. Then we have for $0\leq s\leq1$
\begin{equation}
\| f g \|_{H^{s,p}(X)} \lesssim \| f \|_{H^{s,p_1}(X)} \| g \|_{L^{p_2}(X)} + \| f \|_{L^{p_3}(X)}\| g \|_{H^{s,p_4}(X)}
\end{equation}
where $1<p,p_j<\infty$ with $j=1,\cdots 4$ such that $\frac1p=\frac1{p_1}+\frac1{p_2}=\frac1{p_3}+\frac1{p_4}$.
\end{proposition}

\section{$L^2$-estimates and dispersive estimates}
In this section, we prove the $L^2$-estimates for $U_{j,k}(t)$ and dispersive
estimates for $U_{j,k}(t)U_{j,k}^*(s)$ where $U_{j,k}(t)$ is a micro-localized Klein-Gordon propagator.  The $L^2$-estimate is showed by the spectral
theory on Hilbert space. The conjugate points are separated in the
microlocalized propagators, and hence we can prove
the $TT^*$ version dispersive estimates. Since the abstract Klein-Gordon propagator $U(t)=e^{it\sqrt{1-\Delta_g}}$  behaviors likely 
the Schr\"odinger at low frequency and likes the wave at high frequency, we need to establish dispersive estimate by using different arguments at different frequency. \vspace{0.2cm}

\subsection{Microlocalized propagator}
We start by dividing the Klein-Gordon propagator into a low-energy
piece and a high-energy piece. Using the dyadic partition of unity $1=\sum_{k\in\Z}\phi(2^{-k}\lambda)$
we further define
\begin{equation}
\begin{split}
U_k(t) &= \int_0^\infty e^{it\sqrt{1+\lambda^2}}
\phi(2^{-k}\lambda) dE_{\sqrt{\mathrm{H}}}(\lambda),\quad k\in\Z
\end{split}
\end{equation}
Further using scattering psedodifferential operator partition of
identity operator in Proposition \ref{prop:localized spectral measure}, we define
\begin{equation}\label{Uij}
\begin{gathered}
U_{j,k}(t) = \int_0^\infty e^{it\sqrt{1+\lambda^2}}
\phi(2^{-k}\lambda) Q_j(\lambda)dE_{\sqrt{\mathrm{H}}}(\lambda),\quad 1\leq j\leq N, k\in\Z.
\end{gathered}
\end{equation}
We divide the microlocalized Klein-Gordon propagator into low frequency and high frequency 
\begin{equation}\label{Ulh}
\begin{gathered}
U_{j}^{\mathrm{low}}(t) = \int_0^\infty e^{it\sqrt{1+\lambda^2}}
\phi_0(\lambda) Q_j(\lambda)dE_{\sqrt{\mathrm{H}}}(\lambda),\quad 1\leq j\leq N;\\
U_{j}^{\mathrm{high}}(t) = \sum_{k=0}^\infty\int_0^\infty e^{it\sqrt{1+\lambda^2}}
\phi(2^{-k}\lambda) Q_j(\lambda)dE_{\sqrt{\mathrm{H}}}(\lambda),\quad 1\leq j\leq N.
\end{gathered}
\end{equation}

\subsection{$L^2$-estimate for $U_{j,k}(t)$.} In this subsection we show this definition is well-defined and prove $U_{j,k}(t)$ is
a bounded operator on $L^2(X)$. Essentially this has been proved in \cite[Proposition 3.2]{Zhang}. For convenience, we sketch it here. Indeed it suffices
to show the above integrals are well defined over any compact
interval in $(0, \infty)$.  Let $A(\lambda) = e^{it\sqrt{1+\lambda^2}}
 \phi(2^{-k})Q_j(\lambda)$. Then
$A(\lambda)$ is a family of
bounded operators on $L^2(X)$, compactly supported in $[2^{k-1},2^{k+1}]$
and $\mathcal{C}^1$ in $\lambda\in (0,\infty)$. Integrating by parts, the
integral of
$$
\int_{2^{k-1}}^{2^{k+1}} A(\lambda) dE_{\sqrt{\mathrm{H}}}(\lambda)
$$
is given by
\begin{equation}\label{mean}
E_{\mathrm{\sqrt{\mathrm{H}}}}(2^{k+1}) A(2^{k+1}) - E_{\mathrm{\sqrt{\mathrm{H}}}}(2^{k-1})
A(2^{k-1}) - \int_{2^{k-1}}^{2^{k+1}} \frac{d}{d\lambda} A(\lambda)
E_{\sqrt{\mathrm{H}}}(\lambda) \, d\lambda.
\end{equation}
Hence the operators $U_{j,k}(t)$ are
well-defined by using the following lemma which is the consequence of
\cite[Lemma 2.3, Lemma 3.1]{HZ}.

\begin{lemma}\label{QQ'}
Each $Q_j(\lambda)$ and each operator $\lambda
\partial_\lambda Q_j(\lambda)$ is bounded on $L^2(X)$
uniformly in $\lambda$.\end{lemma}
Since $\|U_{j,k}\|_{L^2\to L^2}\leq C$ is 
equivalent to $\|U_{j,k}U^*_{j,k}\|_{L^2\to L^2}\leq C$,  we compute by  \cite[Lemma 5.3]{HZ},
\begin{equation}\begin{gathered}
U_{j,k}(t) U_{j,k}(t)^* = \int   \phi\big(
\frac{\lambda}{2^k} \big) \phi\big( \frac{\lambda}{2^k} \big)
Q_j(\lambda) dE_{\sqrt{\mathrm{H}}}(\lambda) Q_j(\lambda)^* \\
= -\int \frac{d}{d\lambda} \Big(  \phi\big(
\frac{\lambda}{2^k} \big) \phi\big( \frac{\lambda}{2^k} \big)
Q_j(\lambda) \Big) E_{\sqrt{\mathrm{H}}}(\lambda) Q_j(\lambda)^*  \\
- \int \phi\big( \frac{\lambda}{2^k} \big)
\phi\big( \frac{\lambda}{2^k} \big) Q_j(\lambda)
E_{\sqrt{\mathrm{H}}}(\lambda) \frac{d}{d\lambda}
Q_j(\lambda)^*.
\end{gathered}\label{Uijk}\end{equation}

We observe that this is independent of $t$ and we also note that the
integrand is a bounded operator on $L^2$, with an operator bound of
the form $C/\lambda$ where $C$ is uniform, as we see from
Lemma~\ref{QQ'} and the support property of $\phi$. The integral
is therefore uniformly bounded, as we are integrating over a dyadic
interval in $\lambda$. Hence we have shown that
\begin{proposition}[$L^2$-estimates]\label{energy} Let $U_{j,k}(t)$ be defined in \eqref{Uij}.
Then there exists a constant $C$ independent of $t, z, z'$ such that
$\|U_{j,k}(t)\|_{L^2\rightarrow L^2}\leq C$ for all $j\geq 1,
k\in\Z$.
\end{proposition}

Since there is no difference between $e^{it\lambda^2}$ and $e^{it\sqrt{1+\lambda^2}}$ in the proof \cite[Proposition 5.1]{HZ} (using a almost orthogonal property in 
the summation of $k$), 
we have
\begin{proposition}[$L^2$-estimates]\label{energy'} Let $U^{\mathrm{low}}_{j}(t)$ be defined in \eqref{Ulh}.
Then there exists a constant $C$ independent of $t, z, z'$ such that
$\|U^{\mathrm{low}}_{j}(t)\|_{L^2\rightarrow L^2}\leq C$ for all $j\geq 1$.
\end{proposition}

\subsection{Dispersive estimates} 

In this subsection, we use stationary phase argument and Proposition
\ref{prop:localized spectral measure} to establish the
microlocalized dispersive estimates. Before doing this, we prove a 
fundamental result on decay estimate.

\begin{proposition}[Microlocalized dispersive estimates for low frequency]\label{dispersive-l}
Let $Q_j(\lambda)$ be in Proposition \ref{prop:localized spectral measure}. Then for all integers
$j\geq1$, the kernel estimate
\begin{equation}\label{disper}
\Big|\int_0^\infty e^{it\sqrt{1+\lambda^2}} \phi_0(\lambda) \big(Q_j(\lambda)
dE_{\sqrt{\mathrm{H}}}(\lambda)Q_j^*(\lambda)\big)(z,z')
d\lambda\Big|\leq C (1+|t|)^{-\frac{n}2}
\end{equation}
holds for a constant $C$ independent of points $z,z'\in X$.
\end{proposition}
\begin{proof} The key things in the proof are to use the property of spectral measure in Proposition
\ref{prop:localized spectral measure} and stationary phase argument. When $|t|\lesssim 1$, it is easy to show it due to the compact support of $\phi_0$. From now on, we only need to consider 
the case $t\gg1$ by symmetry.  Let $r=d(z,z')$
and $\bar{r}=rt^{-\frac12}$. In this case, we write the kernel using
Proposition \ref{prop:localized spectral measure}
\begin{equation}\label{4.4?}
\begin{split}
&\int_0^\infty e^{it\sqrt{1+\lambda^2}} \phi_0(\lambda)\big(Q_j(\lambda)
dE_{\sqrt{\mathrm{H}}}(\lambda)Q_j^*(\lambda)\big)(z,z')
d\lambda\\
&=\sum_\pm \int_0^\infty e^{it\sqrt{1+\lambda^2}}e^{\pm
ir\lambda}\lambda^{n-1}\phi_0(\lambda) a_\pm(\lambda,z,z')d\lambda+\int_0^\infty
e^{it\sqrt{1+\lambda^2}}\lambda^{n-1}\phi_0(\lambda)b(\lambda,z,z')d\lambda  \\&= t^{-\frac
n2}\sum_{\pm} \int_0^{\infty} e^{i\sqrt{t^2+t\lambda^2}}e^{\pm
i\bar{r}\lambda}\lambda^{n-1}\phi_0(t^{-1/2}\lambda)a_\pm(t^{-1/2}\lambda,z,z')d\lambda
\\&\quad+\int_0^\infty e^{it\sqrt{1+\lambda^2}}\lambda^{n-1}\phi_0(\lambda)b(\lambda,z,z')d\lambda,
\end{split}
\end{equation}
where $a_\pm$ satisfies estimates
\begin{equation*}
\big|\partial_\lambda^\alpha a_\pm(\lambda,z,z') \big|\leq C_\alpha
\lambda^{-\alpha}(1+\lambda d(z,z'))^{-\frac{n-1}2},
\end{equation*}
and therefore
\begin{equation}\label{beans0}
\Big|\partial_\lambda^\alpha \big(a_\pm(t^{-1/2}\lambda,z,z')\big)
\Big|\leq C_\alpha \lambda^{-\alpha}(1+\lambda
\bar{r})^{-\frac{n-1}2}.
\end{equation}

First, we show the contribution of the above term with $b(\lambda,z,z')$.
We can use the estimate \eqref{beans} to obtain
\begin{equation}\label{4.2}
\begin{split}
\Big|\big(\frac{d}{d\lambda}\big)^{N}b(\lambda,z,z')\Big|\leq
C_N\lambda^{n-1-N}\quad \forall N\in\mathbb{N}.
\end{split}
\end{equation}
Let $\delta$ be a small constant to be chosen later. Recall that we chose $\phi\in
C_c^\infty([\frac12,2])$ such that $\sum_{m \in \Z}\phi(2^{-m}\lambda)=1$; we
denote $\phi_0(\lambda)=\sum_{m\leq -1}\phi(2^{-m}\lambda)$. Then
\begin{equation*}
\Big|\int_0^\infty e^{it\sqrt{1+\lambda^2}} b(\lambda,z,z')
\phi_0(\lambda)\phi_0(\frac{\lambda}{\delta})d\lambda\Big|\leq
C\int_0^\delta\lambda^{n-1}d\lambda\leq C\delta^n.
\end{equation*}
 We use  integration by parts $N$ times to obtain, using  \eqref{4.2}
\begin{equation*}
\begin{split}
&\Big|\int_0^\infty e^{it\sqrt{1+\lambda^2}}
\phi_0(\lambda)\sum_{m\geq0}\phi\big(\frac{\lambda}{2^m\delta}\big)b(\lambda,z,z')
d\lambda\Big|\\
&\leq \sum_{m\geq 0}\Big|\int_0^\infty
\big(\frac{\sqrt{1+\lambda^2}}{i\lambda
t}\frac\partial{\partial\lambda}\big)^{N}\big(e^{it\sqrt{1+\lambda^2}}\big)
\phi_0(\lambda)
\phi\big(\frac{\lambda}{2^m\delta}\big)b(\lambda,z,z')
d\lambda\Big|\\& \leq
C_N|t|^{-N}\sum_{m\geq0}\int_{2^{m-1}\delta}^{2^{m+1}\delta}\lambda^{n-1-2N}d\lambda\leq
C_N|t|^{-N}\delta^{n-2N}.
\end{split}
\end{equation*}
Choosing $\delta=|t|^{-\frac12}$, we have thus proved
\begin{equation}\label{4.3}
\begin{split}
&\Big|\int_0^\infty e^{it\sqrt{1+\lambda^2}} \phi_0(\lambda) b(\lambda,z,z')
d\lambda\Big|\leq C_N|t|^{-\frac n2}.
\end{split}
\end{equation}

 Now we consider first term in RHS
of \eqref{4.4?}. We divide it into two pieces using the partition of
unity above. It suffices to prove that there exists a constant $C$
independent of $\bar{r}$ such that
\begin{equation*}
\begin{split}
I^\pm:=&\Big|\int_0^{\infty} e^{i\sqrt{t^2+t\lambda^2}}e^{\pm
i\bar{r}\lambda}\lambda^{n-1}\phi_0(t^{-1/2}\lambda)a_\pm(t^{-1/2}\lambda,z,z')\phi_0(\lambda)d\lambda\Big|\leq
C,\\II^\pm:=& \Big|\sum_{m\geq0}\int_0^{\infty} e^{i\sqrt{t^2+t\lambda^2}}e^{\pm
i\bar{r}\lambda}\lambda^{n-1} \phi_0(t^{-1/2}\lambda) a_\pm(t^{-1/2}\lambda,z,z')\phi(\frac{\lambda}{2^m})d\lambda\Big|\leq
C.
\end{split}
\end{equation*}
The estimate for $I^\pm$ is obvious, since $\lambda\leq 1$.
For $II^{+}$, we use integration by parts. Notice that
$$
L^+ (e^{i\sqrt{t^2+t\lambda^2} + i \bar{r} \lambda}) =  e^{i\sqrt{t^2+t\lambda^2} + i \bar{r} \lambda}, \quad L^+ = \frac{-i}{\frac{t\lambda}{\sqrt{t^2+t\lambda^2}} + \bar{r}} \frac{\partial}{\partial \lambda}.
$$
Note that if $0<\lambda<\sqrt{t}$, we have for $k\geq0$ by induction
\begin{equation}\label{I-bp4}
\begin{split}
\partial^k_\lambda\left[\left(\frac{t\lambda}{\sqrt{t^2+t\lambda^2}}+\bar{r}\right)^{-1}\right]\leq C_k \lambda^{-1-k}.
\end{split}
\end{equation}
Writing
$$
e^{i\sqrt{t^2+t\lambda^2} + i \bar{r} \lambda}= (L^+)^N (e^{i\sqrt{t^2+t\lambda^2} + i \bar{r} \lambda})
$$
and integrating by parts, we gain a factor of $\lambda^{-2N}$ thanks to \eqref{beans0} and \eqref{I-bp4}.
Thus $II^{+}$ can be estimated by
$$
\sum_{m \geq 0} \int_{\lambda \sim 2^m} \lambda^{n-1-2N} \, d\lambda \leq C.
$$

To treat $II^-$, we introduce a further decomposition, based on the size of $\bar{r} \lambda$. We write
$II^-=II^-_1+II^-_2$, where (dropping the $-$ superscripts and subscripts from here on)
\begin{equation*}
\begin{split}
II_1=&\Big|\sum_{m\geq0}\int_0^{\infty} e^{i\sqrt{t^2+t\lambda^2}}e^{-
i\bar{r}\lambda}\lambda^{n-1}\phi_0(t^{-1/2}\lambda) a(t^{-1/2}\lambda,z,z')\phi(\frac{\lambda}{2^m})
\phi_0(8\bar{r} \lambda) d\lambda\Big|, ~\\II_2=&\Big|\int_0^{\infty}
e^{i\sqrt{t^2+t\lambda^2}}e^{-
i\bar{r}\lambda}\lambda^{n-1}\phi_0(t^{-1/2}\lambda)a(t^{-1/2}\lambda,z,z')
\left(1-\phi_0(\lambda)\right) \big( 1 - \phi_0(8\bar{r} \lambda)
\big) d\lambda\Big|.
\end{split}
\end{equation*}
Let $\Phi(\lambda,\bar{r})=\sqrt{t^2+t\lambda^2}-\bar{r}\lambda$. We first
consider $II_1$. Since the integral for $II_1$ is supported where $\lambda \leq (8 \bar{r})^{-1}$ and $\lambda \geq
1/2$, the integrand is only
nonzero when $\bar{r}\leq1/4$. Since $\lambda<\sqrt{t}$, therefore $|\partial_\lambda\Phi| =
\frac{t\lambda}{\sqrt{t^2+t\lambda^2}} - \bar{r} \geq \frac{\sqrt{2}}2\lambda- \bar{r} \geq\frac1{10}\lambda$. Define the operator
$L=L(\lambda,\bar{r})=(\frac{t\lambda}{\sqrt{t^2+t\lambda^2}}-\bar{r} )^{-1}\partial_\lambda$.  On the support of $\phi_0(\lambda/\sqrt{t})$, we have for $k\geq0$
\begin{equation}\label{I-bp4'}
\begin{split}
\partial^k_\lambda\left[\left(\frac{t\lambda}{\sqrt{t^2+t\lambda^2}}-\bar{r}\right)^{-1}\right]\leq C_k \lambda^{-1-k}.
\end{split}
\end{equation}
By
\eqref{beans0} and using integration by parts, we obtain for $N>n/2$
\begin{equation*}
\begin{split}
II_1\leq&\sum_{m\geq0}\Big|\int_0^{\infty}
e^{i\sqrt{t^2+t\lambda^2}}e^{-
i\bar{r}\lambda}\lambda^{n-1}\phi_0(t^{-1/2}\lambda)a(t^{-1/2}\lambda,z,z')\phi(\frac{\lambda}{2^m})\phi_0(8\bar{r} \lambda) d\lambda\Big|
\\=&\sum_{m\geq0}\Big|\int_0^{\infty} L^{N}
\big(e^{i(\frac{\lambda}{\sqrt{t^2+t\lambda^2}}-
\bar{r}\lambda)}\big)\Big[\lambda^{n-1}\phi_0(t^{-1/2}\lambda)a(t^{-1/2}\lambda,z,z')\phi(\frac{\lambda}{2^m})\phi_0(8\bar{r}
\lambda) \Big]d\lambda\Big|
\\\leq &C_N\sum_{m\geq0}\int_{|\lambda|\sim
2^{m}}\lambda^{n-1-2N}d\lambda\leq C_N.
\end{split}
\end{equation*}
Finally we consider $II_2$. Here, we replace the decomposition $\sum_m \phi(2^{-m} \lambda)$ with a different decomposition,
based on the size of $\partial_\lambda \Phi$.
\begin{equation*}
\begin{split}
II_2\leq &\Big|\int_0^{\infty} e^{i\sqrt{t^2+t\lambda^2}}e^{-
i\bar{r}\lambda}\lambda^{n-1}\phi_0(t^{-1/2}\lambda) a(t^{-1/2}\lambda,z,z')\\ &\qquad \qquad \big(1-\phi_0(\lambda)\big)\phi_0(\frac{t\lambda}{\sqrt{t^2+t\lambda^2}}-\bar{r}) \big( 1 - \phi_0(8\bar{r} \lambda) \big) \, d\lambda\Big|\\
&+\sum_{m\geq0}\Big|\int_0^{\infty}
e^{i\sqrt{t^2+t\lambda^2}}e^{-
i\bar{r}\lambda}\lambda^{n-1}\phi_0(t^{-1/2}\lambda) a(t^{-1/2}\lambda,z,z')
\\&\qquad \qquad \big(1-\phi_0(\lambda)\big)\phi\big(\frac{\frac{t\lambda}{\sqrt{t^2+t\lambda^2}}-\bar{r}}{2^m}\big)\big( 1 - \phi_0(8\bar{r} \lambda) \big) \, d\lambda\Big|\\:=&II_2^1+II_2^2.
\end{split}
\end{equation*}
If $\bar{r} \leq 10$, note $\lambda<\sqrt{t}$ again, then for the integrand of $II_2^1$ to be
nonzero we must have $\lambda \leq 100$, due to the second $\phi_0$ factor in $II_2^1$.
Then it is easy to see that $II_2^1$ is uniformly bounded. If $\bar{r} \geq
10$, by $|\frac{t\lambda}{\sqrt{t^2+t\lambda^2}}-\bar{r}|\leq 1$ and $\lambda<\sqrt{t}$, we have $\bar{r}\sim\lambda$.
Hence, using \eqref{beans0} with $\alpha = 0$,
\begin{equation*}
\begin{split}
II_2^1&\le\int_{\{\lambda<\sqrt{t}:|\frac{t\lambda}{\sqrt{t^2+t\lambda^2}}-\bar{r}|\leq 1\}}\lambda^{n-1}(1+\bar{r}\lambda)^{-\frac{n-1}2}d\lambda
\\&\leq C t^{1/2} \int_{\{\lambda<1:|\frac{\lambda}{\sqrt{1+\lambda^2}}-\frac{\bar{r}}{\sqrt{t}}|\leq 1/\sqrt{t}\}}d\lambda
\\&\leq C t^{1/2} \int_{\{\lambda<1:|\bar{\lambda}-\frac{\bar{r}}{\sqrt{t}}|\leq 1/\sqrt{t}\}}(1+\lambda^2)^{3/2}d\bar{\lambda}\leq C.
\end{split}
\end{equation*}

Now we consider the second term. We write \begin{equation*}
\begin{split}
II_2^2\leq&\sum_{m\geq0}\Big|\int_0^{\sqrt{t}}
e^{i\sqrt{t^2+t\lambda^2}}e^{-
i\bar{r}\lambda}\phi_0(t^{-1/2}\lambda)\lambda^{n-1}a(t^{-1/2}\lambda,z,z')
\\&\qquad\qquad\big(1-\phi_0(\lambda)\big)\phi(\frac{\frac{t\lambda}{\sqrt{t^2+t\lambda^2}}-\bar{r}}{2^m})\big( 1 - \phi_0(8\bar{r} \lambda) \big) \, d\lambda\Big|
\\=&\sum_{m\geq0}\Big|\int L^{N}
\big(e^{i(\sqrt{t^2+t\lambda^2}-
\bar{r}\lambda)}\big)\Big[\phi_0(t^{-1/2}\lambda)\lambda^{n-1}a(t^{-1/2}\lambda,z,z')\\&\qquad \qquad\big(1-\phi_0(\lambda)\big)\phi(\frac{\frac{t\lambda}{\sqrt{t^2+t\lambda^2}}-\bar{r}}{2^m})\big( 1 - \phi_0(8\bar{r} \lambda) \big) \Big] \, d\lambda\Big|.
\end{split}
\end{equation*}
Let $$b(\lambda)=\lambda^{n-1}a(t^{-1/2}\lambda,z,z')\big(1-\phi_0(\lambda)\big)\phi(\frac{\frac{t\lambda}{\sqrt{t^2+t\lambda^2}}-\bar{r}}{2^m})\big( 1 - \phi_0(8\bar{r} \lambda) \big),$$
then we have the rough estimate, due to the support of $b$
$$|\partial_\lambda ^\alpha b|\leq C_{\alpha}\lambda^{n-1} (1+\bar{r}\lambda)^{-(n-1)/2}.$$
Hence we obtain
\begin{equation*}
\begin{split}
|(L^*)^N [b(\lambda)]|&\leq C_N 2^{-mN}\lambda^{n-1} (1+\bar{r}\lambda)^{-(n-1)/2}.
\end{split}
\end{equation*}
Therefore we obtain by using integrating by parts and
\eqref{beans0}
\begin{equation*}
\begin{split}
II_2^2 \leq C_N\sum_{m\geq0}2^{-mN}\int_{\{\lambda<\sqrt{t},|\frac{t\lambda}{\sqrt{t^2+t\lambda^2}}-\bar{r}|\sim 2^m\}}\lambda^{n-1}(1+\bar{r}\lambda)^{-\frac{n-1}2}d\lambda.
\end{split}
\end{equation*}
If $\bar{r}\leq 2^{m+1}$, then $\lambda \leq 2^{m+2}$ on the support
of the integrand. 
\begin{equation*}
\begin{split}
II_2^2\leq C_N\sum_{m\geq0}2^{-mN}2^{(m+2)n}\leq C.
\end{split}
\end{equation*}
If $\bar{r}\geq 2^{m+1}$, we
have $\lambda\sim \bar{r}$, thus
\begin{equation*}
\begin{split}
II_2^2\leq C_N t^{1/2}  \sum_{m\geq0}2^{-mN} \int_{\{\lambda<1:|\frac{\lambda}{\sqrt{1+\lambda^2}}-\frac{\bar{r}}{\sqrt{t}}|\sim \frac{2^m}{\sqrt{t}}\}}d\lambda \leq C_N\sum_{m\geq0}2^{-mN}2^{m},
\end{split}
\end{equation*}
which is summable
for $N > 1$.  Therefore we have completed the proof of Proposition
\ref{dispersive-l}.\end{proof}\vspace{0.2cm}

\begin{proposition}[Microlocalized dispersive estimates for high frequency]\label{dispersive-h}
Let $Q_j(\lambda)$  be in Proposition \ref{prop:localized spectral measure}. Then for all integers
$j\geq1$ and $k\geq0$, the kernel estimate 
\begin{equation}\label{disper}
\begin{split}
\Big|\int_0^\infty e^{it\sqrt{1+\lambda^2}} \phi(2^{-k}\lambda) \big(Q_j(\lambda)
&dE_{\sqrt{\mathrm{H}}}(\lambda)Q_j^*(\lambda)\big)(z,z')
d\lambda\Big|\\ &\leq C 2^{k(n+1+\theta)/2}\left(2^{-k}+|t|\right)^{-(n-1+\theta)/2}.
\end{split}
\end{equation}
holds for $0\leq \theta\leq 1$ and a constant $C$ independent of $k$ and points $z,z'\in X$.
\end{proposition}

\begin{proof} 
Let $h=2^{-k}\leq 1$. The key to the proof is to use the estimates in Proposition
\ref{prop:localized spectral measure}.  If $|t|\leq h$, it is easy to see \eqref{disper} due to 
\begin{equation*}
\begin{split}
\Big|Q_j(\lambda)
dE_{\sqrt{\mathrm{H}}}(\lambda)Q_j^*(\lambda)
\Big|\leq C \lambda^{n-1}.
\end{split}
\end{equation*}
From now on, we only consider $|t|\geq h=2^{-k}$. By the scaling, this is a directly consequence of
\begin{equation}\label{disper'}
\begin{split}
\Big|\int_0^\infty e^{it\sqrt{h^2+\lambda^2}/h} \phi(\lambda) &\big(Q_j
dE_{\sqrt{\mathrm{H}}}Q_j^*\big)(\lambda/h, z,z')
d\lambda\Big|\\&\leq C h^{-(n-1)}(|t|/h)^{-\frac{n-1}2}(1+h|t|)^{-1/2}.
\end{split}
\end{equation}
Indeed if we have done this, we have for $0\leq\theta\leq1$
\begin{equation*}
\begin{split}
&\Big|\int_0^\infty e^{it\sqrt{1+\lambda^2}} \phi(2^{-k}\lambda) \big(Q_j(\lambda)
dE_{\sqrt{\mathrm{H}}}(\lambda)Q_j^*(\lambda)\big)(z,z')
d\lambda\Big|\\ &\leq C 2^{k(n+1)/2}|t|^{-(n-1)/2}\left(1+2^{-k}|t|\right)^{-1/2}\\&\leq C 2^{k(n+1+\theta)/2}(2^{-k}+|t|)^{-(n-1+\theta)/2}(2^{-k}|t|)^{\frac{\theta}2}\left(1+2^{-k}|t|\right)^{-1/2}
\end{split}
\end{equation*}
which implies \eqref{disper}.

Now we prove \eqref{disper'}.  Let $r=d(z,z')$,  we write
\begin{equation}\label{4.4}
\begin{split}
&\Big|\int_0^\infty e^{it\sqrt{h^2+\lambda^2}/h} \phi(\lambda) \big(Q_j
dE_{\sqrt{\mathrm{H}}}Q_j^*\big)(\lambda/h, z,z')
d\lambda\\
&=\sum_\pm \int_0^\infty e^{it\sqrt{h^2+\lambda^2}/h}e^{\pm
ir\lambda/h}\phi(\lambda)(\lambda/h)^{n-1}a_\pm(\lambda/h,z,z')d\lambda\\&\qquad\qquad+\int_0^\infty
e^{it\sqrt{h^2+\lambda^2}/h}\phi(\lambda)(\lambda/h)^{n-1}b(\lambda/h,z,z')d\lambda
\end{split}
\end{equation}
where $a_\pm$ satisfies estimates
\begin{equation*}
\big|\partial_\lambda^\alpha a_\pm(\lambda,z,z') \big|\leq C_\alpha
\lambda^{-\alpha}(1+\lambda d(z,z'))^{-\frac{n-1}2},
\end{equation*}
and therefore
\begin{equation}\label{beans0}
\Big|\partial_\lambda^\alpha \big(a_\pm(h^{-1}\lambda,z,z')\big)
\Big|\leq C_\alpha \lambda^{-\alpha}(1+h^{-1}\lambda
r)^{-\frac{n-1}2}.
\end{equation}

Consider the terms  with
the `$b$' term, then we can use the estimate \eqref{beans} to obtain
\begin{equation}
\begin{split}
\Big|\big(\frac{d}{d\lambda}\big)^{N}\big(\phi(\lambda)(\lambda/h)^{n-1}b(\lambda/h,z,z')\big)\Big|\leq
C_N(\lambda/h)^{n-1}\lambda^{-N},\quad \forall N\in\mathbb{N}.
\end{split}
\end{equation}
Let $\delta$ be a small constant to be chosen later. Recall that we chose $\phi\in
C_c^\infty([\frac12,2])$ such that $\sum_{m \in \Z}\phi(2^{-m}\lambda)=1$; we
denote $\phi_0(\lambda)=\sum_{m\leq -1}\phi(2^{-m}\lambda)$. Then
\begin{equation*}
\begin{split}
\Big|\int_0^\infty e^{it\sqrt{h^2+\lambda^2}/h} \phi(\lambda)(\lambda/h)^{n-1}&b(\lambda/h,z,z')
\phi_0(\frac{\lambda}{\delta})d\lambda\Big|\\&\leq
C\int_0^{\delta}(\lambda/h)^{n-1}d\lambda\leq C h(\delta/h)^n.
\end{split}
\end{equation*}
 We use  integration by parts $N$ times to obtain, using  \eqref{4.2},
\begin{equation*}
\begin{split}
&\Big|\int_0^\infty e^{it\sqrt{h^2+\lambda^2}/h}
\sum_{m\geq0}\phi(\frac{\lambda}{2^{m}\delta})\phi(\lambda)(\lambda/h)^{n-1}b(\lambda/h,z,z')
d\lambda\Big|\\
&\leq \sum_{m\geq 0}\Big|\int_0^\infty
\big(\frac{h\sqrt{h^2+\lambda^2}}{\lambda
t}\frac\partial{\partial\lambda}\big)^{N}\big(e^{it\sqrt{h^2+\lambda^2}/h}\big)
\phi(\frac{\lambda}{2^m\delta})\phi(\lambda)(\lambda/h)^{n-1}b(\lambda/h,z,z')
d\lambda\Big|\\& \leq
C_N(|t|/h)^{-N} h^{-(n-1)}\sum_{m\geq0}\int_{2^{m-1}\delta}^{2^{m+1}\delta}\lambda^{n-1-2N} d\lambda\leq
C_N(|t|/h)^{-N}h^{-(n-1)}\delta^{n-2N}.
\end{split}
\end{equation*}

Choosing $\delta=(|t|/h)^{-\frac12}$ and noting $|t|\geq h$, we have thus proved 
\begin{equation}
\begin{split}
&\Big|\int_0^\infty e^{it\sqrt{h^2+\lambda^2}/h} \phi(\lambda)(\lambda/h)^{n-1}b(\lambda/h,z,z')
d\lambda\Big|\\&\leq C h (h|t|)^{-\frac{n}2}\leq C(h|t|)^{-\frac{n-1}2}(h^{-1}|t|)^{-1/2}\leq C (|t|h)^{-\frac{n-1}2}(1+h|t|)^{-1/2}.
\end{split}
\end{equation}

Next we consider the terms with $a_\pm$. Without loss of generality, we consider $t\gg h$. Let $\Phi_{\pm}(\lambda, h, r,t)=\sqrt{h^2+\lambda^2}\pm \frac{\lambda r}{t}$,  
it suffices to show there exists a constant $C$ independent of $r, t$ and $h$ such that
\begin{equation}\label{disper''}
\begin{split}
|I_h^\pm(t,r)|\leq C (|t|/h)^{-\frac{n-1}2}(1+h|t|)^{-1/2}
\end{split}
\end{equation}
where
\begin{equation}
\begin{split}
I_h^\pm(t,r):=\int_0^\infty e^{i\frac{t}{h}\Phi_\pm(\lambda, h, r, t)}\phi(\lambda)\lambda^{n-1}a_\pm(\lambda/h,z,z')d\lambda.
\end{split}
\end{equation}
If $r<t/4$ or $r>2t$, a simpler computation gives
$$|\partial_\lambda \Phi_\pm(\lambda, h, r, t)|=\left|\frac{\lambda}{\sqrt{h^2+\lambda^2}}\pm\frac r t\right|\geq 1/4. $$

It is not difficult to use the Leibniz rule to prove 
\begin{lemma} Let $L=(\frac {it}h\partial_\lambda\Phi)^{-1}\partial_\lambda$ and let $L^*$ be its adjoint operator. Suppose 
that $b(\lambda)$ satisfies $$|\partial_\lambda^\alpha b(\lambda)|\leq \lambda^{n-1-|\alpha|}.$$
Then we have for any $N\geq 0$
\begin{equation}\label{I-bp1}
|(L^*)^N [b(\lambda)]|\leq C \lambda^{n-1-N} \sum_{j=0}^N\frac{(t/h)^j}{\left|\frac {it}h\partial_\lambda\Phi\right|^{N+j}}.
\end{equation}
\end{lemma}
By integrating by parts and using this lemma, we obtain for $r<\frac t4$ or $r>2t$
\begin{equation}
\begin{split}
|I_h^\pm(t,r)|\leq C (|t|/h)^{-N},\quad \forall N\geq 0
\end{split}
\end{equation}
which implies \eqref{disper''} since $t\geq h$. Therefore we only need consider  the case $t\sim r$.
A rought estimate gives
\begin{equation}\label{rough}
\begin{split}
|I_h^\pm(t,r)|\leq \int_0^\infty \phi(\lambda)\lambda^{n-1}(1+\lambda r/h)^{-(n-1)/2} d\lambda \leq C (|t|/h)^{-\frac{n-1}2}.
\end{split}
\end{equation}
Note that 
$$|\partial_\lambda \Phi_+(\lambda, h, r, t)|=\left|\frac{\lambda}{\sqrt{h^2+\lambda^2}}+\frac r t\right|\geq 1/2, $$
by using the same stationary phase argument again, we obtain
\begin{equation}
\begin{split}
|I_h^+(t,r)|\leq C (|t|/h)^{-N},\quad \forall N\geq 0.
\end{split}
\end{equation}
To estimate $I_h^-(t,r)$, we need the following Van der Corput lemma, see \cite{Stein}
\begin{lemma}[Van der Corput] Let $\phi$ be real-valued and smooth in $(a,b)$, and that $|\phi^{(k)}(x)|\geq1$ for all $x\in (a,b)$. Then
\begin{equation}
\left|\int_a^b e^{i\lambda\phi(x)}\psi(x)dx\right|\leq c_k\lambda^{-1/k}\left(|\psi(b)|+\int_a^b|\psi'(x)|dx\right)
\end{equation}
holds when (i) $k\geq2$ or (ii)$k=1$ and $\phi'(x)$ is monotonic. Here $c_k$ is a constant depending only on $k$.
\end{lemma}

It is easy to check for $h\leq 1$ and $\lambda\sim 1$
$$|\partial^2_\lambda \Phi_-(\lambda, h, r, t)|=\left|\frac{h^2}{\sqrt{h^2+\lambda^2}}\right|\geq \frac{h^2}{100}. $$
By using the Van der Corput lemma with $\lambda=th$, we show
\begin{equation}
\begin{split}
|I_h^-(t,r)|&\leq C (|t|/h)^{-1/2}\int_0^\infty \left|\frac{d}{d\lambda}\left(\phi(\lambda)\lambda^{n-1}a_{-}(\lambda/h,z,z')\right)\right| d\lambda\\
&\leq C (|t|/h)^{-1/2}\int_0^2\lambda^{n-2} (1+\lambda r/h)^{-\frac{n-1}2} d\lambda\leq C(th)^{-1/2}(t/h)^{-\frac{n-1}2}.
\end{split}
\end{equation}
This together with \eqref{rough}, we prove \eqref{disper''}.
\end{proof}

As two consequences of Proposition \ref{dispersive-l} and Proposition \ref{dispersive-h} respectively, we
immediately have
\begin{proposition}\label{prop:Dispersive-l} Let $U^{\mathrm{low}}_{j}(t)$ be defined in \eqref{Ulh}.
Then there exists a constant $C$ independent of $t, z, z'$ for all 
$j\geq 1$, such that
\begin{equation}\label{Dispersive}
\|U^\mathrm{low}_{j}(t)(U_{j}^\mathrm{low})^*(s)\|_{L^1\rightarrow L^\infty}\leq C
(1+|t-s|)^{-n/2}.
\end{equation}
\end{proposition}

\begin{proposition}\label{prop:Dispersive-h} Let $U_{j,k}(t)$ be defined in \eqref{Uij}.
Then there exists a constant $C$ independent of $t, z, z'$ for all 
$j\geq 1, k\in\Z^+$ such that
\begin{equation}\label{Dispersive}
\|U_{j,k}(t)U^*_{j,k}(s)\|_{L^1\rightarrow L^\infty}\leq C
2^{k(n+1+\theta)/2}(2^{-k}+|t-s|)^{-(n-1+\theta)/2}
\end{equation}
where $0\leq \theta\leq 1$.
\end{proposition}

\section{Strichartz estimates}
In this section, we show the Strichartz estimates in Theorem
\ref{Strichartz}. To obtain the Strichartz estimates for high frequency, we need a
variant of Keel-Tao's abstract Strichartz estimate.
\subsection{Semiclassical Strichartz estimates}
We recall a variety of the abstract Keel-Tao's Strichartz estimates
theorem proved in \cite{Zhang}, which is an analogue of the semiclassical Strichartz
estimates for Schr\"odinger in \cite{KTZ, Zworski}.

\begin{proposition}\label{prop:semi}
Let $(X,\mathcal{M},\mu)$ be a $\sigma$-finite measured space and
$U: \mathbb{R}\rightarrow B(L^2(X,\mathcal{M},\mu))$ be a weakly
measurable map satisfying, for some constants $C$, $\alpha\geq0$,
$\sigma, h>0$,
\begin{equation}\label{md}
\begin{split}
\|U(t)\|_{L^2\rightarrow L^2}&\leq C,\quad t\in \mathbb{R},\\
\|U(t)U(s)^*f\|_{L^\infty}&\leq
Ch^{-\alpha}(h+|t-s|)^{-\sigma}\|f\|_{L^1}.
\end{split}
\end{equation}
Then for every pair $q,r\in[1,\infty]$ such that $(q,r,\sigma)\neq
(2,\infty,1)$ and
\begin{equation*}
\frac{1}{q}+\frac{\sigma}{r}\leq\frac\sigma 2,\quad q\ge2,
\end{equation*}
there exists a constant $\tilde{C}$ only depending on $C$, $\sigma$,
$q$ and $r$ such that
\begin{equation}\label{stri}
\Big(\int_{\R}\|U(t) u_0\|_{L^r}^q dt\Big)^{\frac1q}\leq \tilde{C}
\Lambda(h)\|u_0\|_{L^2}
\end{equation}
where $\Lambda(h)=h^{-(\alpha+\sigma)(\frac12-\frac1r)+\frac1q}$.
\end{proposition}

\subsection{Homogeneous Strichartz estimates} Now we prove the homogeneous Strichartz estimates. Using the
Littlewood-Paley frequency cutoff $\phi_m(\sqrt{\mathrm{H}})$, we
define
\begin{equation}\label{loc}
u_m(t,\cdot)=\phi_m(\sqrt{\mathrm{H}})u(t,\cdot).
\end{equation}
Then the frequency localized solutions
$\{u_m\}_{m\in\Z}$ solves the family of Cauchy problems
\begin{equation}\label{leq}
\partial_{t}^2u_m+\mathrm{H} u_m+u_m=0, \quad u_m(0)=f_m(z),
~\partial_tu_m(0)=g_m(z),
\end{equation}
where $f_m=\phi_m(\sqrt{\mathrm{H}})u_0$ and
$g_m=\phi_m(\sqrt{\mathrm{H}})u_1$. 
Then we can write the solution 
\begin{equation}
u=u^{l}+u^{h},  \qquad u^l=\sum_{m\leq-1}u_m, \quad u^h=\sum_{m\geq0}u_m
\end{equation}
Let $U(t)=e^{it\sqrt{1+\mathrm{H}}}$, then
we write
\begin{equation}\label{sleq}
\begin{split}
u_m(t,z)
=\frac{U(t)+U(-t)}2f_m+\frac{U(t)-U(-t)}{2i\sqrt{1+\mathrm{H}}}g_m.
\end{split}
\end{equation}
Notice that
\begin{equation*}
U(t)=\sum_{j=1}^{N}\sum_{k\in\Z}U_{j,k}(t)=\sum_{j=1}^{N}U_{j}^{\mathrm{low}}(t)+\sum_{j=1}^{N}\sum_{k\geq0}U_{j,k}(t),
\end{equation*}
we can write
\begin{equation*}
U(t)f=\sum_{j}\sum_{k\in\mathbb{Z}}\int_0^\infty
e^{it\sqrt{1+\lambda^2}}\phi(2^{-k}\lambda)Q_j(\lambda)dE_{\sqrt{\mathrm{H}}}(\lambda)
\widetilde{\phi}(2^{-k}\sqrt{\mathrm{H}})f
\end{equation*}
where $\widetilde{\phi} \in C_0^\infty(\R\setminus\{0\})$ takes
values in $[0,1]$ such that $\widetilde{\phi}\phi=\phi$. In
view of  $f_m=\phi(2^{-m}\sqrt{\mathrm{H}})f$, then
$\widetilde{\phi}(2^{-k}\sqrt{\mathrm{H}})f_m$ vanishes if
$|m-k|\geq3$. 
Then we have
\begin{equation}\label{LU(t)}
U(t)f_m=\sum_{j}\sum_{|k-m|\leq 3}\int_0^\infty
e^{it\sqrt{1+\lambda^2}}\phi(2^{-k}\lambda)Q_j(\lambda)dE_{\sqrt{\mathrm{H}}}(\lambda)f_m.
\end{equation}
By the squarefunction
estimates \eqref{square} and Minkowski's inequality, we obtain for
$q,r\geq2$
\begin{equation}\label{LP}
\|u\|_{L^q(\R;L^r(X))}\lesssim \|u^l\|_{L^q(\R;L^r(X))}+
\Big(\sum_{m\geq0}\|u_m\|^2_{L^q(\R;L^r(X))}\Big)^{\frac12}.
\end{equation}

To prove the homogeneous estimates in Theorem \ref{Strichartz}, that
is $F=0$, we need
\begin{proposition}\label{lStrichartz} Let $f_m=\phi_m(\sqrt{\mathrm{H}})u_0$, we have for $m\geq 0$
\begin{equation}\label{lstri}
\|U(t)f_m\|_{L^q_tL^r_z(\mathbb{R}\times X)}\lesssim
2^{ms}\|f_m\|_{L^2(X)},
\end{equation}
where the K-G admissible pair $(q,r)\in [2,\infty]^2$ and $s$ satisfy
\eqref{adm} and \eqref{scaling}.
\end{proposition}
Indeed, by using Proposition \ref{energy'},
Proposition \ref{prop:Dispersive-l} and the argument in Keel-Tao \cite{KT}, we have for $2/q\leq n(1/2-1/r)$
\begin{equation}
\begin{split}
\|U_j^{\mathrm{low}} u_0\|_{L^q(\R;L^r(X))}\leq C\|u_0\|_{L^2(X)}.
\end{split}
\end{equation}
Without loss generality, we assume $u_1=0$. By using the Proposition \ref{lStrichartz}, we have
\begin{equation}
\begin{split}
&\|u^l\|_{L^q(\R;L^r(X))}\leq C\sum_{j=1}^N\|U_j^{\mathrm{low}} u_0\|_{L^q(\R;L^r(X))}\leq C\|u_0\|_{L^2(X)},\\
&\sum_{m\geq0}\|u_m\|^2_{L^q(\R;L^r(X))}\leq C 2^{2ms}\|f_m\|^2_{L^2(X)}\leq C\|u_0\|^2_{H^s(X)}.
\end{split}
\end{equation}
Therefore we prove the Strichartz estimate with $u_1=F=0$
\begin{equation}
\|u\|_{L^q(\R;L^r(X))}\leq C\|u_0\|_{H^s(X)}.
\end{equation}

Now we prove this proposition. By using Proposition \ref{energy} and
Proposition \ref{prop:Dispersive-h}, we have the estimates \eqref{md}
for $U_{j,k}(t)$, where $\alpha=(n+1+\theta)/2$, $\sigma=(n-1+\theta)/2$ and
$h=2^{-k}$. Then it follows from Proposition \ref{prop:semi} that
\begin{equation*}
\|U_{j,k}(t)f_m\|_{L^q_t(\R:L^r(X))}\lesssim
2^{k[(n+\theta)(\frac12-\frac1r)-\frac1q]} \|f_m\|_{L^2(X)}.
\end{equation*}
By \eqref{LU(t)}, we obtain 
\begin{equation*}
\|U(t)f_m\|_{L^q_t(\R:L^r(X))}\lesssim
2^{m[(n+\theta)(\frac12-\frac1r)-\frac1q]} \|f_m\|_{L^2(X)}=2^{ms} \|f_m\|_{L^2(X)}
\end{equation*}
which proves \eqref{lstri}.

\subsection{Inhomogeneous Strichartz estimates}
In this subsection, we prove the inhomogeneous Strichartz estimates. Let
$U(t)=e^{it\sqrt{1+\mathrm{H}}}: L^2\rightarrow L^2$. We have already
proved that
\begin{equation}
\|U(t)u_0\|_{L^q_tL^r_z}\lesssim\|u_0\|_{H^s}
\end{equation} holds for all $(q,r,s)$ satisfying \eqref{adm} and \eqref{scaling}.
For $s\in\R$ and $(q,r)$ satisfying \eqref{adm} and \eqref{scaling},
we define the operator $T_s$ by
\begin{equation}\label{Ts}
\begin{split}
T_s: L^2_z&\rightarrow L^q_tL^r_z,\quad f\mapsto (1+\mathrm{H})^{-\frac
s2}e^{it\sqrt{1+\mathrm{H}}}f.
\end{split}
\end{equation}
Then we have by duality
\begin{equation}\label{Ts*}
\begin{split}
T^*_{1-s}: L^{\tilde{q}'}_tL^{\tilde{r}'}_z\rightarrow L^2,\quad
F(\tau,z)&\mapsto \int_{\R}(1+\mathrm{H})^{\frac
{s-1}2}e^{-i\tau\sqrt{1+\mathrm{H}}}F(\tau)d\tau,
\end{split}
\end{equation}
where $1-s=n(\frac12-\frac1{\tilde{r}})-\frac1{\tilde{q}}$.
Therefore we obtain
\begin{equation*}
\Big\|\int_{\R}U(t)U^*(\tau)\mathrm{H}^{-\frac12}F(\tau)d\tau\Big\|_{L^q_tL^r_z}
=\big\|T_sT^*_{1-s}F\big\|_{L^q_tL^r_z}\lesssim\|F\|_{L^{\tilde{q}'}_tL^{\tilde{r}'}_z}.
\end{equation*}
Since $s=n(\frac12-\frac1r)-\frac1q$ and
$1-s=n(\frac12-\frac1{\tilde{r}})-\frac1{\tilde{q}}$, thus $(q,r),
(\tilde{q},\tilde{r})$ satisfy \eqref{scaling}. By the
Christ-Kiselev lemma \cite{CK}, we thus obtain for $q>\tilde{q}'$,
\begin{equation}\label{non-inhomgeneous}
\begin{split}
\Big\|\int_{\tau<t}\frac{\sin{(t-\tau)\sqrt{1+\mathrm{H}}}}
{\sqrt{1+\mathrm{H}}}F(\tau)d\tau\Big\|_{L^q_tL^r_z}\lesssim\|F\|_{L^{\tilde{q}'}_t{L}^{\tilde{r}'}_z}.
\end{split}
\end{equation}
Notice that for all $(q,r), (\tilde{q},\tilde{r})$ satisfy
\eqref{adm} and \eqref{scaling}, we must have $q>\tilde{q}'$.
Therefore we have proved all inhomogeneous Strichartz estimates
including $q=2$.

\section{Wellposedness and small nonlinear scattering}
In this section, we prove Theorem \ref{thm1} and Theorem \ref{thm2}. We prove the results by
a contraction mapping argument. The key point is the application of
Strichartz estimates. 

\subsection{Proof of Theorem \ref{thm1} }
Let $q_0=(n+1)(p-1)/2$, $q_1=2(n+1)/(n-1)$ and
$\alpha=s_c-\frac12$. For any small constant $\epsilon>0$ such that
\begin{equation}
\begin{split}Y:=\Big\{u: ~&u\in C_t(H^{s_c})\cap L^{q_0}([0,T];L^{q_0}(X))\cap L^{q_1}([0,T];
H^{\alpha}_{q_1}(X)),\\&
\|u\|_{L^{q_0}([0,T];L^{q_0}(X))}+\|u\|_{L^{q_1}([0,T];
H^{\alpha}_{q_1}(X))}\leq
C\epsilon\Big\}.\end{split}\end{equation} Consider the solution map
$\Phi$ defined by
\begin{equation*}
\begin{split}
\Phi(u)&=\cos(t\sqrt{1+\mathrm{H}})u_0(z)+\frac{\sin(t\sqrt{1+\mathrm{H}})}{\sqrt{1+\mathrm{H}}}u_1(z)
+\int_0^t\frac{\sin\big((t-s)\sqrt{1+\mathrm{H}}\big)}{\sqrt{1+\mathrm{H}}}F(u(s,z))\mathrm{d}s
\\&=:u_{\text{hom}}+u_{\text{inh}},
\end{split}
\end{equation*}
where $F(u)=\pm |u|^{p-1}u$. We claim the map $\Phi: Y\rightarrow
Y$ is contracting. We first note that the Sobolev
embedding $ L^{q_0}_tH^{\alpha}_{r_0}\hookrightarrow
L_{t,z}^{q_0}$ where $r_0=(\frac{\alpha}n+\frac1{q_0})^{-1}$.  Since $p\geq1+4/(n-1),$ thus $s_c\geq 1/2$. On the other hand, it is easy to
check that the pairs $(q_0,r_0), (q_1,q_1)$ satisfy \eqref{adm} and
\eqref{scaling} with $s=1/2$ and $\theta=0$. By Theorem \ref{Strichartz}, we obtain 
\begin{equation}\label{5.4'}
\begin{split}
\|u_{\text{hom}}\|_{C_t(H^{s_c})\cap
L^{q_0}(\R;L^{q_0}(X))\cap L^{q_1}(\R;
H^{\alpha}_{q_1}(X))}\leq C\big(\|u_0\|_{
H^{s_c}}+\|u_1\|_{H^{s_c-1}}\big).
\end{split}
\end{equation}
Hence we must have
\begin{equation}\label{5.4}
\begin{split}
\|u_{\text{hom}}\|_{ L^{q_0}([0,T];L^{q_0}(X))\cap
L^{q_1}([0,T];H^{\alpha}_{q_1}(X))}\leq \frac12C\epsilon
\end{split}
\end{equation}
for $T=\infty$ if the initial data has small norm $\epsilon(p)$, or,
if not, this inequality will be satisfied for some $T>0$ by the
dominated convergence theorem.  Applying Theorem \ref{Strichartz} with
$\tilde{q}'=\tilde{r}'=\frac{2(n+1)}{n+3}$, one has
\begin{equation}\label{5.5}
\begin{split}
\|u_{\text{inh}}\|_{C_t(H^{s_c})\cap
L^{q_0}([0,T];L^{q_0}(X))\cap L^{q_1}([0,T];
H^{\alpha}_{q_1}(X))}\leq C\|F(u)\|_{L^{\tilde{q}'}_t
H^{\alpha}_{\tilde{r}'}}.
\end{split}
\end{equation}
Note $p\in [1+\frac{4}{n-1}, 1+\frac{4}{n-2}]$, we have $0\leq \alpha\leq1$. By using the
fraction Liebniz rule for Sobolev spaces in Proposition \ref{frule}, we have
\begin{equation}\label{5.5}
\begin{split}
\|F(u)\|_{L^{\tilde{q}'}_tH^{\alpha}_{\tilde{r}'}}\leq
C\|u\|^{p-1}_{L^{q_0}_{t,z}}\|u\|_{L^{q_1}_t
H^{\alpha}_{q_1}}\leq C^2(C\epsilon)^{p-1}\epsilon\leq
\frac{C\epsilon}2.
\end{split}
\end{equation}
A similar argument as above leads to
\begin{equation}\label{5.6}
\begin{split}
&\|\Phi(u_1)-\Phi(u_2)\|_{L^{q_1}([0,T];
H^{\alpha}_{q_1}(X))\cap
L^{q_0}([0,T];L^{q_0}(X))}\\&\leq
C\|F(u_1)-F(u_2)\|_{L^{\tilde{q}'}_t
H^{\alpha}_{\tilde{r}'}}\\&\leq
C^2(C\epsilon)^{p-1}\|u_1-u_2\|_{L^{q_1}([0,T];
H^{\alpha}_{q_1}(X))\cap
L^{q_0}([0,T];L^{q_0}(X))}\\&\leq
\frac12\|u_1-u_2\|_{L^{q_1}([0,T];
H^{\alpha}_{q_1}(X))\cap L^{q_0}([0,T];L^{q_0}(X))}.
\end{split}
\end{equation}
Therefore the solution map $\Phi$ is a contraction map on $Y$ under
the metric $d(u_1,u_2)=\|u_1-u_2\|_{{L^{q_1}([0,T];
H^{\alpha}_{q_1}(X))}\cap L^{q_0}([0,T];L^{q_0}(X))}$.
The standard contraction argument proves the first part of Theorem
\ref{thm1}. Noting that the above argument needs the condition $p\in [1+\frac{4}{n-1}, 1+\frac{4}{n-2}]$ in \eqref{5.4'}. If $(u_0,u_1)\in H^1(X)\times L^2(X)$, we 
extend the local well-posedness for $p\in [p_{\text{conf}}, 1+\frac{4}{n-2}]$. By energy conservation law, we obtain the global existence for large data and finish the final part of  Theorem
\ref{thm1}.

\subsection{Proof of Theorem \ref{thm2} }

For a constant $C$ we define
\begin{equation}
\begin{split}\widetilde{Y}:=\Big\{u: ~&\|u\|_{C_t([0,T];H^{1+\delta})\cap
L^{2}([0,T];L^{\infty}(X))}\leq 2C\Big\}.\end{split}\end{equation} Consider the solution map
$\Phi$ defined by
\begin{equation*}
\begin{split}
\Phi(u)&=\cos(t\sqrt{1+\mathrm{H}})u_0(z)+\frac{\sin(t\sqrt{1+\mathrm{H}})}{\sqrt{1+\mathrm{H}}}u_1(z)
+\int_0^t\frac{\sin\big((t-s)\sqrt{1+\mathrm{H}}\big)}{\sqrt{1+\mathrm{H}}}F(u(s,z))\mathrm{d}s
\\&=:u_{\text{hom}}+u_{\text{inh}},
\end{split}
\end{equation*}
where $F(u)$ is replaced by $F(u)=uDu+|u|^2u$. By Theorem \ref{Strichartz} with $0<\theta=2\delta\ll1$, we obtain
\begin{equation}
\begin{split}
\|u_{\text{hom}}\|_{C_t([0,T];H^{1+\delta})\cap
L^{2}([0,T];L^{\infty}(X))}\leq C\big(\|u_0\|_{
H^{1+\delta}}+\|u_1\|_{H^{\delta}}\big).
\end{split}
\end{equation}
Furthermore one has by Theorem \ref{Strichartz} and choosing small $T$
\begin{equation*}
\begin{split}
&\|u_{\text{inh}}\|_{C_t([0,T];H^{1+\delta})\cap
L^{2}([0,T];L^{\infty}(X))}\\&\leq C\|u D u\|_{L^{\frac1{1-\delta}}_t
L^2(X)}+C\|u^3\|_{L^{\frac1{1-\delta}}_t
L^2(X)}\\ &\leq C T^{\frac12-\delta}\|u\|_{ L^{2}([0,T];L^{\infty}(X))}\|D u\|_{L^{\infty}_t([0,T];
L^2(X))}+CT^{(1-\delta)/3}\|u\|^3_{L^{\infty}_t
L^6(X)}\leq 2C.
\end{split}
\end{equation*}
By choosing $T$ small enough, we have
\begin{equation*}
\begin{split}
&\|\Phi(u_1)-\Phi(u_2)\|_{C_t([0,T];H^{1+\delta})\cap
L^{2}([0,T];L^{\infty}(X))}\\ &\leq C T^{\frac14}\left(\|u_1-u_2\|_{ L^{2}([0,T];L^{\infty}(X))}+\|D (u_1-u_2)\|_{L^{\infty}_t([0,T];
L^2(X))}+\|u_1-u_2\|_{L^{\infty}_t
L^6(X)}\right)\\&\leq \frac12 \|u_1-u_2\|_{C_t([0,T];H^{1+\delta})\cap
L^{2}([0,T];L^{\infty}(X))}
\end{split}
\end{equation*}
The standard contraction argument on $\widetilde{Y}$ completes the proof of Theorem
\ref{thm2}.

\begin{center}

\end{center}
\end{document}